\documentclass[11pt]{article}

\usepackage{fullpage}
\usepackage[utf8]{inputenc}
\usepackage[T1]{fontenc}
\usepackage{amsmath,amsfonts,amssymb,amsthm,bbm}
\usepackage{stackrel}
\usepackage{enumerate}
\usepackage{graphicx}
\usepackage{hyperref}
\usepackage{subfigure}
\usepackage{color}

\newtheorem{theorem}{Theorem}

\newtheorem{lemma}[theorem]{Lemma}
\newtheorem{definition}[theorem]{Definition}

\newtheorem{remark}[theorem]{Remark}

\numberwithin{theorem}{section}
\numberwithin{equation}{section}
\numberwithin{figure}{section}

\newcommand{\ve}{\varepsilon}

\newcommand{\ol}{\overline}
\newcommand{\ul}{\underline}

\newcommand{\calF}{\mathcal{F}}

\newcommand{\calC}{\mathcal{C}}
\newcommand{\calA}{\mathcal{A}}
\newcommand{\calO}{\mathcal{O}}
\newcommand{\calJ}{\mathcal{J}}

\newcommand{\ZZ}{\mathbb{Z}}
\newcommand{\RR}{\mathbb{R}}
\newcommand{\CC}{\mathbb{C}}
\newcommand{\HH}{\mathbb{H}}

\newcommand{\PP}{\mathbb{P}}

\newcommand{\EE}{\mathbb{E}}

\newcommand{\TT}{\mathbb{T}} 
\newcommand{\Ball}{B} 
\newcommand{\din}{\partial^{\textrm{in}}} 
\newcommand{\dout}{\partial^{\textrm{out}}} 
\newcommand{\cluster}{\mathcal{C}}
\newcommand{\Cinf}{\cluster_\infty} 
\newcommand{\Ch}{\mathcal{C}_H} 
\newcommand{\Cv}{\mathcal{C}_V} 
\newcommand{\arm}{\mathcal{A}} 

\newcommand{\lra}{\leftrightarrow}

\begin{document}

\title{Two-dimensional forest fires with boundary ignitions}
\author{Jacob van den Berg\footnote{CWI, Amsterdam; E-mail: \texttt{J.van.den.Berg@cwi.nl}.}, Pierre Nolin\footnote{City University of Hong Kong; E-mail: \texttt{bpmnolin@cityu.edu.hk}. Partially supported by a GRF grant from the Research Grants Council of the Hong Kong SAR (project CityU11309323).}}

\date{}

\maketitle

\begin{abstract}
In the classical Drossel-Schwabl forest fire process, vertices of a lattice become occupied at rate $1$, and they are hit by lightning at some tiny rate $\zeta > 0$, which causes entire connected components to burn. In this paper, we study a variant where fires are coming from the boundary of the forest instead.

In particular we prove that, for the case without recoveries where the forest is an $N \times N$ box in the triangular lattice, the probability that the center of the box gets burnt tends to $0$ as $N \rightarrow \infty$ (but substantially slower than the one-arm probability of critical Bernoulli percolation). And, for the case where the forest is the upper-half plane, we show (still for the version without recoveries) that no infinite occupied cluster emerges. We also discuss analogs of some of these results for the corresponding models with recoveries, and explain how our results and proofs give valuable insight on a process considered earlier by Graf \cite{Gr2014}, \cite{Gr2016}.

\bigskip


\end{abstract}


\section{Introduction} \label{sec:intro}

\subsection{Background and motivation} \label{sec:backgr}

Let $G = (V,E)$ be an infinite two-dimensional lattice, such as the square lattice $\ZZ^2$ or the triangular lattice $\TT$, where $V$ and $E$ contain its vertices and edges, respectively. We consider processes which are indexed by time $t \in [0,\infty)$, and consist of vertex configurations $(\omega_v(t))_{v \in V_D}$ in given subdomains $D = (V_D,E_D)$ of $G$. Here, $V_D \subseteq V$, and $E_D$ contains all edges in $E$ whose endpoints both lie in $V_D$. At each time $t \in [0,\infty)$, every vertex $v \in V_D$ can be in three possible states: \emph{vacant} ($\omega_v(t) = 0$), \emph{occupied} ($\omega_v(t) = 1$), and \emph{burnt} ($\omega_v(t) = -1$).

Initially, at time $t=0$, all vertices in $V_D$ are vacant, and they then become occupied at rate $1$. Moreover, we add the following \emph{ignition} mechanism along the outer boundary $\dout V_D$ of $V_D$ (consisting of all the vertices $v \in V \setminus V_D$ which are neighbors of one -- or several -- vertices in $V_D$, i.e., such that $\{v,v'\} \in E$ for some $v' \in V_D$): each vertex in $\dout V_D$ is hit by lightning at some given rate $\zeta \in (0,\infty]$. When this happens, all occupied neighbors in $V_D$ of this vertex, as well as all the vertices connected (in $V_D$) to these neighbors by an occupied path, become burnt immediately. Note that we allow the rate $\zeta$ to be infinite, which corresponds to connected components of occupied vertices (also called \emph{occupied clusters}, or simply clusters) burning as soon as they touch the boundary.

We can then adopt two natural rules for the future evolution of these burnt vertices, leading to two distinct processes: either they remain burnt (state $-1$) forever (forest fire \emph{without} recovery), or they are simply considered the same as vacant, so that they can become occupied again, still at rate $1$ (forest fire \emph{with} recovery) -- and then burn again, and so on. Observe that when there are no recoveries, every vertex $v$ remains eventually (at times sufficiently large, depending on $v$) either occupied or burnt.

In this paper, we analyze these processes in two particular situations. First, we consider sequences of increasing finite subdomains $D_N = (V_N,E_N)$, $N \geq 1$. We then analyze forest fire processes in the upper half-plane, where $V_D = V \cap \HH$, denoting $\HH := \{(x,y) \in \RR^2 \: : \: y \geq 0\}$. In this case, ignitions thus originate along the real line. This specific setting is directly connected to earlier papers by Graf \cite{Gr2014}, \cite{Gr2016}, that were an important inspiration for us.

In the absence of ignitions, we would get the classical Bernoulli site percolation process in $V_D$, at each time $t \in [0,\infty)$. In this process, each vertex $v \in V_D$ is either occupied or vacant, with respective probabilities $p$ and $1-p$, independently of the other vertices, where $p = p(t) := 1 - e^{-t}$ is the percolation parameter. In what follows, we refer to this underlying percolation process as the \emph{pure birth} process. On the whole (2D) lattice $G$, Bernoulli percolation is known to display a phase transition at some distinguished value of $p$, called the percolation threshold, that we denote by $p_c^{\textrm{site}}(G)$ ($\in (0,1)$): for each $p < p_c^{\textrm{site}}(G)$ (subcritical regime), there exists (almost surely) no infinite cluster, while there is (at least) one for $p > p_c^{\textrm{site}}(G)$ (supercritical regime). In particular, we can consider the \emph{critical} time $t_c := -\log (1-p_c^{\textrm{site}}(G))$, at which the pure birth configuration is exactly critical (and so an infinite cluster starts to emerge).

Our main goal in the present work is to understand the macroscopic effect of the boundary ignitions, and compare them to ``bulk'' ignitions, as in the classical Drossel-Schwabl forest fire process \cite{DrSc1992}. For such processes, existence was established by D\"urre \cite{Du2006}, in any dimension $d \geq 2$, and further properties were derived in recent works \cite{BN2021}, \cite{LN2021}. Note that strictly speaking, these two latter papers are concerned with the variant without recovery, but there is strong reason to believe that the near-critical behavior, close to time $t_c$ (and in fact, slightly later than that), is essentially the same when recoveries are allowed, as we emphasized in \cite{BN2022} (in a different, but related setting). For boundary ignitions, proving existence requires significant work in the upper half-plane, and this was done in \cite{Gr2014} (under an extra condition on the burning mechanism, see Section~\ref{sec:intro_Graf} in the present paper). Clearly, in the case of finite domains the process is a finite-state Markov chain and hence existence is standard.

\subsection{Main results} \label{sec:main_res}

Let us now describe informally our results. From now on, we focus on the triangular lattice $\TT$. We need to do it for technical reasons, as it is the two-dimensional lattice on which the most precise results are known rigorously for Bernoulli percolation. As before, we denote by $V$ and $E$ its set of vertices and set of edges, respectively.

First, we consider the forest fire process in finite ``hexagonal'' subdomains of $\TT$: for each $N \geq 1$, we let $H_N$ be the domain consisting of all the vertices within a graph distance $N$ from the origin $0$ (see Figure~\ref{fig:hexagon} below for an illustration). In order to emphasize the dependence on $N$, we use the notation $\PP_N$.

\begin{theorem} \label{thm:main_result1}
Consider the forest fire process without recoveries in $H_N$, $N \geq 1$, with boundary ignitions at a given rate $\zeta \in (0, \infty]$. We have
\begin{equation} \label{eq:main}
\PP_N(0 \text{ is eventually burnt}) \stackrel{N \to \infty}{\longrightarrow} 0.
\end{equation}
Moreover, for any $\delta > 0$: for all $N$ sufficiently large,
\begin{equation} \label{eq:main_quant}
\PP_N(0 \text{ is eventually burnt}) \geq N^{-\frac{5}{52} - \delta}.
\end{equation}
\end{theorem}

In particular, note that the left-hand inequality of \eqref{eq:main_quant} implies that for the event that $0$ gets burnt at some time, its probability vanishes substantially more slowly than $\PP_{p_c}(0 \lra \partial H_N)$. Here, $0 \lra \partial H_N$ denotes the existence of an occupied connection to the boundary of the domain (in other words, the event that $0$ belongs to an occupied cluster that reaches $\partial H_N$). Indeed, $\PP_{p_c}(0 \lra \partial H_N)$ is known to decay as $N^{-\frac{5}{48} + o(1)}$ from \cite{LSW2002} (see the explanation between \eqref{eq:L_equiv} and \eqref{eq:theta_L_exp} below). By following closely the successive steps in our proof of \eqref{eq:main}, and estimating each of the (high probability) events along the way (one of them needs to fail), one can also obtain an explicit power-law upper bound on the probability that $0$ burns. However, so far we could only get such a bound with an exponent smaller than the exponent $\frac{5}{52}$ in the lower bound.

\begin{remark}
For our results, we do not need ignitions to come from the whole boundary: it is enough that only a positive fraction (bounded away from $0$ as $N \to \infty$) of the vertices get ignited. For example, exactly the same result would hold in the case of ignitions along the bottom side of $H_N$ only. In other words, we could state them in $H_N \cap \HH$, with ignitions coming from vertices with $y$-coordinate $-\frac{\sqrt{3}}{2} (N+1)$, and $N \to \infty$. Our reasonings do not really require that $\zeta$ is constant either, and we could allow it to be a function $\zeta_N$ of $N$. In this case, we need to require that $\zeta_N$ does not tend $0$ too quickly, and the relevant condition for our proofs is $\zeta_N \gg N^{-2/3}$ as $N \to \infty$.
\end{remark}

\begin{figure}[t]
\begin{center}

\includegraphics[width=.82\textwidth]{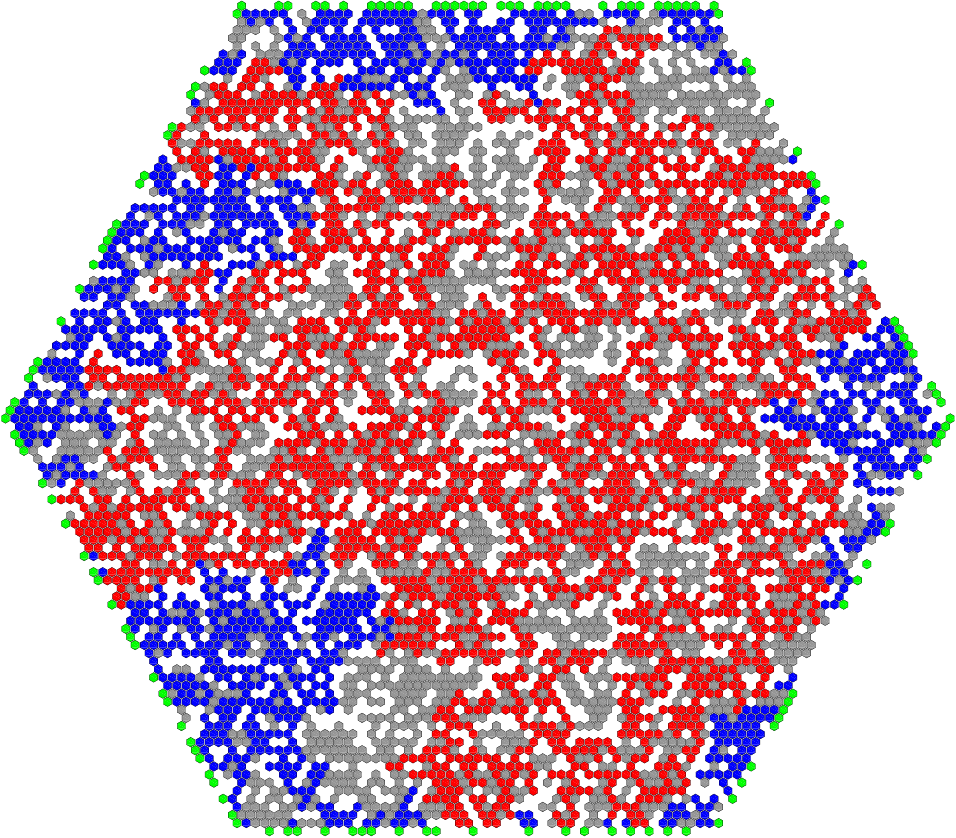}
\caption{\label{fig:simulation} This figure illustrates a typical configuration, at time $t = 2t_c$, for the forest fire process with boundary ignitions. More specifically, it depicts this process in the (finite) hexagonal domain $H_N$, with $N = 50$ and ignition rate $\zeta = \frac{1}{2}$. The vertices ignited so far, along $\dout H_N$, are shown in green. The white vertices are vacant, the red cluster is the largest burnt cluster, while the other burnt vertices are colored in blue, and the occupied (unburnt) vertices in gray.}

\end{center}
\end{figure}

Our reasonings are based on the existence of paths entering deep into the domain by staying within a cone, with an opening angle strictly smaller than $\pi$. These paths are used to control precisely the spread of ignitions from the boundary. We develop these geometric considerations in Section~\ref{sec:cone_sites}, showing that ``many'' such long-distance connections exist.

Roughly speaking, we show that slightly after time $t_c$, some portions of the boundary remain available, from where ignitions can spread. These boundary arcs, when ignited, allow one large cluster to burn (shown in red on Figure~\ref{fig:simulation}). This cluster is macroscopic, and it burns at a time which is slightly supercritical (more precisely, the characteristic length for the pure birth process tends to $\infty$ as $N \to \infty$, but it is still $\ll N$). We will use it to deduce that with high probability, this giant cluster does not reach the origin, but rather, it surrounds it and remains at a mesoscopic distance from it. In particular, the existence of this mesoscopic island containing the origin, unaffected by ignitions, means that the probability for the origin to burn tends to $0$.

As for the process with recoveries, tools originating from a paper by Kiss, Manolescu and Sidoravicius \cite{KMS2015} can then be harnessed, as in our paper \cite{BN2022}, to derive an analogous result in that situation. More precisely, we can control the process up to some time strictly later than the critical time $t_c$.

\begin{theorem} \label{thm:main_result2}
Consider the forest fire process with recoveries in $H_N$, $N \geq 1$, with boundary ignitions at rate $\zeta \in (0, \infty]$. For some universal $\hat{t} > t_c$ (that is, independent of $N$ and $\zeta$), we have
\begin{equation} \label{eq:main2}
\PP_N(0 \text{ gets burnt before time } \hat{t}) \stackrel{N \to \infty}{\longrightarrow} 0.
\end{equation}
\end{theorem}

Furthermore, the same quantitative lower bound holds in this case: for any $\delta > 0$, for all $N$ sufficiently large,
\begin{equation} \label{eq:main_quant2}
\PP_N(0 \text{ gets burnt before time } \hat{t}) \geq N^{-\frac{5}{52} - \delta}.
\end{equation}

We now turn to Graf's forest fire process in the upper half-plane. We consider, rather, a version of this process without recoveries, for which we prove the following result.

\begin{theorem} \label{thm:main_resultH}
Consider the forest fire process without recoveries in $V \cap \HH$, with boundary ignitions at rate $\zeta \in (0, \infty]$. Then almost surely, no infinite occupied cluster emerges. More precisely, with probability $1$, we have: for every vertex $v \in V \cap \HH$, there exists $L < \infty$ such that at all times $t \in [0,\infty)$, the occupied cluster of $v$ contains at most $L$ vertices.
\end{theorem}

We comment further on this theorem, and some of its intuitive implications, in Section~\ref{sec:intro_Graf}. As we explain briefly in Section~\ref{sec:connection_Graf}, we believe that the process with recoveries displays an analogous property, namely that no infinite cluster emerges before some universal time $\hat{t} > t_c$. However, additional technical difficulties arise in that setting, which are related, informally speaking, to the combined effect of recoveries near the real line. We plan to handle them in a future work.

\subsection{Discussion}

\subsubsection{Connections with Graf's model} \label{sec:intro_Graf}

In Graf's model, the forest corresponds to the upper half-plane in the triangular lattice, and there is, besides the ignitions from the boundary, an additional source of burnings (at least, theoretically): namely, as soon as an infinite occupied cluster emerges, it is burnt instantaneously.

This two-types-of-ignition model may look somewhat artificial, but it has a very natural motivation, as explained in the Introduction of \cite{Gr2016}. In the two papers \cite{Gr2014} and \cite{Gr2016}, Graf proved the existence, and several interesting properties, of that process. In particular he showed that a.s., each vertical column contains only finitely many sites that burnt before or at $t_c$, and infinitely many sites that burn after $t_c$.

However, the actual role of the second source of burnings in his model remained quite mysterious. More precisely, the question whether an infinite occupied cluster emerges at all remained open (it is listed as the second open problem in Section~2 of \cite{Gr2014}). For the version without recoveries, our Theorem~\ref{thm:main_resultH} implies a negative answer to that question. Our last paragraph of Section~\ref{sec:main_res} (as well as the first paragraph in Section~\ref{sec:connection_Graf} below) expresses our belief that there is a time $\hat t > t_c$ such that for Graf's process with recoveries, restricted to the time interval $[0,\hat t]$, the answer is negative too.

Theorems~\ref{thm:main_result1} and \ref{thm:main_result2}, besides being of independent interest, also have the following (more indirect) connection with Graf's work. As said in the Introduction of \cite{Gr2016}, Graf's original motivation came from the study of possible subsequential limits of the forest fire process on an $N$-by-$N$ box in the triangular lattice, with ignitions from the boundary of the box, as $N \to \infty$ (and with fixed center for the boxes). Our Theorems~\ref{thm:main_result1} and \ref{thm:main_result2} (or, rather, a simple modification of their proofs) imply that, for the case without recoveries, and for the case with recoveries but where time is restricted to the time interval $[0, \hat t]$, any subsequential limit process is simply the pure birth process.

\subsubsection{Process with impurities}

We want to conclude this introduction by mentioning another possible approach, which was developed in \cite{M2021}. The main result of that thesis, Theorem~4.11, is in some sense a weak version of our Theorem~\ref{thm:main_result1}. It does not imply that in the case we study (i.e. with fixed ignition rates) the probability that the center of the box burns tends to $0$, but that in the case where, roughly speaking, the ignition rates are of the form $N^{-\ve}$, the probability that the center does not burn is at least $C N^{-\ve}$, where $C = C(\ve) > 0$. We believe that a substantial refinement of a so-called arm-stability result involved in his method might lead to an alternative proof of our result \eqref{eq:main}, but that proof would be much more elaborate than that in the present paper.

The method in \cite{M2021} uses a percolation process with random impurities, attached independently along the boundary of the domain. This idea was directly inspired by the analogous process introduced in \cite{BN2021} (for the original Drossel-Schwabl process, i.e. ``bulk'' ignitions). In that paper, independent impurities are deleted from all over the lattice, in order to analyze the effect of fires taking place early (before the pure birth process enters an associated near-critical window). More precisely, the percolation process provides a stochastic lower bound, which can be used to ensure that with sufficient probability, some prescribed occupied (unburnt) connections exist in the forest fire process.

\begin{figure}[t]
\begin{center}

\includegraphics[width=.85\textwidth]{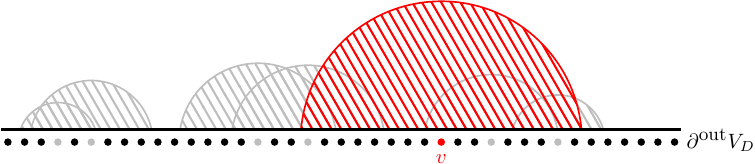}
\caption{\label{fig:impurities} In the process with impurities, each vertex along $\dout V_D$ (with $D$ the upper half-plane in this figure), such as the vertex denoted by $v$ here, can be the center of an impurity, with some given probability. When it is the case, then we remove, i.e. we turn vacant, all the vertices inside $V_D$ within a distance $R_v$ from $v$, where $R_v$ is a random variable with a suitable tail. We then study Bernoulli percolation in the remaining parts of $V_D$.}

\end{center}
\end{figure}

Let $D = (V_D, E_D)$ be a subdomain of $G$. In the process with impurities used in \cite{M2021} (illustrated in Figure~\ref{fig:impurities}), we let each boundary vertex $v \in \dout V_D$ be, independently of the others, the center of a ``hole'' (impurity) with a random radius $R_v$ satisfying, for some constant $c_0$,
$$\PP(R_v \geq k) \leq c_0 k^{-\frac{13}{12} + \ve} \quad (k \geq 1).$$
Note that the exponent $\frac{13}{12}$ already appeared in \cite{Gr2016} (as the sum $\frac{1}{\nu} + \rho$ in the notations of that paper, see Condition~3.1 there), and it will show up again multiple times in the present paper, for roughly the same reasons, notably in Lemma~\ref{lem:long_path}.

\subsection{Organization of the paper} \label{sec:intro_org}

In Section~\ref{sec:percolation}, we first set notations for Bernoulli percolation, focusing on the two-dimensional setting, and we also recall classical properties of that process, which are needed later for our proofs. Next, in Section~\ref{sec:cone_sites}, we develop results about the existence of cone sites, which play a key role in our subsequent reasonings. We then study, in Section~\ref{sec:ff_hexagon}, the forest fire process with boundary ignitions in a finite domain, establishing Theorems~\ref{thm:main_result1} and \ref{thm:main_result2}. Finally, in Section~\ref{sec:connection_Graf}, we analyze the process in the upper half-plane, proving in particular Theorem~\ref{thm:main_resultH}.

\section{Near-critical percolation in 2D} \label{sec:percolation}

In this section, we start by introducing notations for Bernoulli percolation in dimension two, in Section~\ref{sec:notation}. We then recall classical properties of two-dimensional percolation in Section~\ref{sec:classical_prop}, especially at and near its critical point, before stating results which are more specific to percolation in cones, in Section~\ref{sec:perc_cones}.

\subsection{Setting and notations} \label{sec:notation}

Let $\|.\|$ be the usual Euclidean distance in the plane $\RR^2$. In this paper, we work on the triangular lattice $\TT = (V, E)$, with set of vertices $V = V_{\TT} := \big\{ x + y e^{i \pi / 3} \in \CC \: : \: x, y \in \ZZ \big\}$ (we identify $\CC \simeq \RR^2$), and set of edges $E = E_{\TT} := \{ \{v, v'\} \: : \: v, v' \in V \text{ with } \|v - v'\| = 1 \}$. Two vertices $v$, $v' \in V$ such that $\|v - v'\| = 1$, i.e. $\{v, v'\} \in E$, are said to be neighbors, and we use the notation $v \sim v'$. A path is a finite sequence of vertices $v_0, v_1, \ldots, v_k$, for some $k \geq 1$ called the length of the path, such that $v_i \sim v_{i+1}$ for all $i = 0, \ldots, k-1$. Usually, we assume the vertices in this sequence to be distinct, i.e. the path does not use the same vertex twice. A circuit is such a path whose vertices are all distinct, except that $v_0 = v_k$. For $n \geq 0$, we let $B_n := \{ v \in V \: : \: \|v\| < n\}$ be the (open) ball of radius $n$, and for $0 \leq n_1 < n_2$, let $A_{n_1,n_2} := \{ v \in V \: : \: n_1 < \|v\| < n_2\}$ be the annulus with radii $n_1$ and $n_2$ (both centered at the origin $0$). For $v \in V$, we also denote $B_n(v) := v + B_n$ and $A_{n_1,n_2}(v) := v + A_{n_1,n_2}$. Finally, for a subset $A \subseteq V$, we denote $A^c := V \setminus A$, we define its inner (vertex) boundary by $\din A := \{v \in A \: : \: v \sim v'$ for some $v' \in A^c\}$, and its outer boundary by $\dout A := \din (A^c)$ ($= \{v \in A^c \: : \: v \sim v'$ for some $v' \in A\}$).

Bernoulli (site) percolation on the triangular lattice is obtained by tossing a coin for each vertex $v \in V$: for some given value $p \in [0,1]$ (the percolation parameter), $v$ is declared to be occupied or vacant, with respective probabilities $p$ and $1-p$, independently of the other vertices. This produces a percolation configuration $\omega = (\omega_v)_{v \in V}$ belonging to $\Omega := \{0,1\}^V$, which is equipped with the product measure $\PP_p$. In such an $\omega$, we say that two vertices $v, v' \in V$ are connected if we can find an occupied path from $v$ to $v'$. i.e. a path $v_0, v_1, \ldots, v_k$, for some $k \geq 0$, with $v_0 = v$, $v_k = v'$, and consisting only of occupied vertices. This is indicated by $v \lra v'$. Note that in particular, $v$ and $v'$ themselves have to be occupied. Two subsets of vertices $A$, $A' \subseteq V$ are said to be connected, denoted by $A \lra A'$, if we can find two vertices $v \in A$ and $v' \in A'$ which are connected. Occupied vertices can be grouped into connected components (i.e. classes for the equivalence relation $\lra$), that we call occupied clusters. We denote by $\calC(v) := \{v' \in V \: : \: v' \lra v\}$ the occupied cluster containing $v$ (so that $\calC(v) = \emptyset$ when $v$ is vacant). Moreover, we can also define vacant paths and vacant clusters in an analogous way, simply replacing occupied vertices by vacant ones.

Bernoulli percolation displays a phase transition as the parameter $p$ increases, in the following sense. We define $\theta(p) := \PP_p(0 \lra \infty)$, where for each $v \in V$, $v \lra \infty$ denotes the event that $\calC(v)$ is infinite. We also use the notation $\Cinf$ for the union of all infinite occupied clusters. Then the the special value $p_c = p_c^{\textrm{site}}(\TT)$, called the percolation threshold, satisfies the following. For all $p \leq p_c$, there exists no infinite occupied cluster, almost surely (a.s.). On the other hand, for all $p > p_c$, there exists a.s. an infinite cluster, which is furthermore unique. Moreover, it is now classical that $p_c^{\textrm{site}}(\TT) = \frac{1}{2}$ (the proof is an adaptation of \cite{Ke1980}, see Section~3.4 in \cite{Ke1982}). For more detailed background on Bernoulli percolation, the reader can consult the classical references \cite{Ke1982}, \cite{Gr1999}. 

We can consider rectangles on the lattice, which are subsets of the form $R = ([x_1,x_2] \times [y_1,y_2]) \cap V$, for $x_1 < x_2$ and $y_1 < y_2$. In particular, we sometimes use the notation $R_{n_1,n_2} := ( [0,n_1] \times [0,n_2] ) \cap V$ ($\subseteq \HH$) for $n_1, n_2 \geq 0$. Assuming that $R$ is non-empty, we can define in a natural way its left and right sides $\din_L R$ and $\din_R R$ (resp.), which are subsets of $\din R$. By definition, a horizontal (occupied) crossing of $R$ is an occupied path which stays in $R$, and connects $\din_L R$ and $\din_R R$. The event that such a path exists is denoted by $\Ch(R)$, and we can define in an analogous way vertical crossings (connecting the top and bottom sides), using the notation $\Cv(R)$ in this case. Furthermore, considering vacant paths instead, we obtain vacant horizontal and vertical crossings, and we denote their existence by (resp.) $\Ch^*(R)$ and $\Cv^*(R)$. Similarly, for an annulus $A = A_{n_1,n_2}(v)$ of the form above, we consider the event that there exists an occupied (resp. vacant) circuit around $A$, i.e. a circuit which remains in $A$ and surrounds $\Ball_{n_1}(v)$ once, and we denote it by $\calO(A)$ (resp. $\calO^*(A)$).

\subsection{Classical properties} \label{sec:classical_prop}

Our reasonings use primarily techniques and results developed to describe critical and near-critical Bernoulli percolation in dimension two, i.e. for values of the parameter $p$ which are sufficiently close to $p_c$. Indeed, as we explain in the subsequent sections, the relevant macroscopic behavior of our forest fire processes takes place at times when the density of trees in the forest is approximately critical, an instance of the phenomenon of self-organized criticality. More quantitatively, we measure the distance to criticality via the characteristic length $L$, which is defined as
\begin{equation} \label{eq:L_def}
L(p) := \min \big\{ n \geq 1 \: : \: \PP_p \big( \Cv( R_{2n,n} ) \big) \leq 0.001 \big\} \:\:\: \Big( p < \frac{1}{2} \Big), \text{ and } L(p) := L(1-p) \:\:\: \Big( p > \frac{1}{2} \Big).
\end{equation}
The classical Russo-Seymour-Welsh bounds at criticality ($p = \frac{1}{2}$) imply directly that $L(p) \to \infty$ as $p \to \frac{1}{2}$, so it is natural to set $L(\frac{1}{2}) := \infty$.


We make use of the following properties, which are now considered standard (see e.g. \cite{Ke1987} and \cite{No2008}).

\begin{enumerate}[(i)]
\item \emph{Russo-Seymour-Welsh bounds near criticality.} For all $K \geq 1$, there exists $\delta(K) > 0$ such that: for all $p \in (0,1)$, $1 \leq n \leq K L(p)$,
\begin{equation} \label{eq:RSW_bound}
\PP_p \big( \Ch( R_{4n,n} ) \big) \geq \delta \quad \text{and} \quad \PP_p \big( \Ch^*( R_{4n,n} ) \big) \geq \delta.
\end{equation}

\item \emph{Exponential decay property with respect to $L(p)$.} There exist universal constants $c_1, c_2 > 0$ such that: for all $p < \frac{1}{2}$, $n \geq 1$,
\begin{equation} \label{eq:exp_decay_L1}
\PP_p \big( \Cv( R_{4n,n} ) \big) \leq c_1 e^{- c_2 \frac{n}{L(p)}}
\end{equation}
(see Lemma~39 in \cite{No2008}). By duality, we have the following corresponding statement for $p > \frac{1}{2}$: for all $n \geq 1$,
\begin{equation} \label{eq:exp_decay_L2}
\PP_p \big( \Ch( R_{4n,n} ) \big) \geq 1 - c_1 e^{- c_2 \frac{n}{L(p)}}.
\end{equation}


\item \label{it:theta_L} \emph{Asymptotic estimates for $\theta$ and $L$.} Denote by $\pi_1(n)$ (resp. $\pi_4(n)$), $n \geq 0$, the probability at $p = \frac{1}{2}$ that there exists an occupied path (resp. four paths, which are occupied, vacant, occupied, and vacant, in counterclockwise order), connecting $0$ to distance $n$ (resp. each connecting a neighbor of $0$ to distance $n$). We have the following equivalences near $p_c$:
\begin{equation} \label{eq:theta_equiv}
\theta(p) \asymp \pi_1(L(p)) \quad \text{as $p \searrow \frac{1}{2}$,}
\end{equation}
(see Theorem 2 in \cite{Ke1987}, or (7.25) in \cite{No2008}), and
\begin{equation} \label{eq:L_equiv}
\big| p - p_c \big| L(p)^2 \pi_4 \big( L(p) \big) \asymp 1 \quad \text{as $p \to \frac{1}{2}$.}
\end{equation}
(see (4.5) in \cite{Ke1987}, or Proposition 34 in \cite{No2008}). Combined with the values of the one-arm exponent $\alpha_1 = \frac{5}{48}$ \cite{LSW2002} and the (polychromatic) four-arm exponent $\alpha_4 = \frac{5}{4}$ \cite{SW2001} (so that $\pi_j(n) = n^{-\alpha_j + o(1)}$ as $n \to \infty$, for $j=1,4$), these yield the critical exponents for $\theta$ and $L$:
\begin{equation} \label{eq:theta_L_exp}
\theta(p) = \Big( p-\frac{1}{2} \Big)^{\frac{5}{36} + o(1)} \quad \text{as $p \searrow \frac{1}{2}$,} \quad \text{and} \quad L(p) = \Big| p-\frac{1}{2} \Big|^{-\frac{4}{3} + o(1)} \quad \text{as $p \to \frac{1}{2}$.}
\end{equation}
\end{enumerate}

The characteristic length $L$ can be used to define \emph{near-critical parameter scales}, for each $n \geq 1$. These scales are formally defined as
$$p_{\lambda}(n) = \frac{1}{2} + \frac{\lambda}{n^2 \pi_4(n)}, \quad \lambda \in (-\infty,\infty),$$
where $\pi_4(.)$ denotes the four-arm probability at criticality (see (\ref{it:theta_L}) above), and they satisfy: for each fixed $\lambda \neq 0$, $L(p_{\lambda}(n)) \asymp n$ as $n \to \infty$. In particular, using that $\pi_4(n) = n^{-\frac{5}{4} + o(1)}$, we get that for each $\lambda \neq 0$, $p_{\lambda}(n) = \frac{1}{2} + \lambda n^{-\frac{3}{4} + o(1)}$. We will often use these scales implicitly: for example, if we consider the parameter $p_c + n^{-\beta}$, then $L(p_c + n^{-\beta}) \ll n$ or $\gg n$, depending on whether $\beta > \frac{3}{4}$ or $\beta < \frac{3}{4}$.

\subsection{Percolation in cones} \label{sec:perc_cones}

Let $\HH := \RR \times [0,\infty)$ be the closed upper half-plane. We write $V_{\HH} := V \cap \HH$, and we let $E_{\HH}$ be the corresponding set of edges. For $\alpha \in (0,\frac{\pi}{2}]$ and $v \in \din V_{\HH}$ (i.e. with $y$-coordinate equal to $0$), we denote by $\calC^{(\alpha)}(v)$ the intersection of $V_{\HH}$ with the closed cone with apex $v$ and opening angle $2 \alpha$ (see Figure~\ref{fig:cone_site} below). Note that $\calC^{(\alpha)}(v)$ contains $v$ by definition, and this yields a connected subgraph of $(V_{\HH},E_{\HH})$ for $\alpha \geq \frac{\pi}{6}$.

For $n \geq 0$, we consider the truncated cone $\calC_n^{(\alpha)}(v) := \calC^{(\alpha)}(v) \cap B_n(v)$. We denote by $\arm_1(\calC_n^{(\alpha)}(v))$ the event that there exists an occupied path connecting $v$ to $\din B_n(v)$ within the cone. In particular for $\alpha = \frac{\pi}{2}$, we get the usual one-arm event in the upper half-plane (and we need $\alpha \geq \frac{\pi}{6}$ to ensure $\calA_1(\calC_n^{(\alpha)}(v)) \neq \emptyset$). We write
$$\pi_1^{(\alpha)}(n) := \PP_{\frac{1}{2}}\big( \calA_1(\calC_n^{(\alpha)}(0)) \big) \quad (n \geq 0).$$
We also consider the sections $\calC_{n_1,n_2}^{(\alpha)}(v) := \calC^{(\alpha)}(v) \cap A_{n_1,n_2}(v)$, for $0 \leq n_1 < n_2$, and the event $\arm_1(\calC_{n_1,n_2}^{(\alpha)}(v))$ that there exists an occupied path ``crossing'' such a section. Formally, such a path is required to remain in $\calC_{n_1,n_2}^{(\alpha)}(v)$, and to connect two vertices $v_1$ and $v_2$, each having a neighbor outside the section, lying at a distance from $v$ which is, respectively, $\leq n_1$ and $\geq n_2$. We denote by
$$\pi_1^{(\alpha)}(n_1,n_2):= \PP_{\frac{1}{2}}\big( \calA_1(\calC_{n_1,n_2}^{(\alpha)}(0)) \big) \quad (0 \leq n_1 < n_2)$$
the one-arm probabilities at criticality in the cone $\calC^{(\alpha)}(0)$. Finally, as usual, we use the lighter notations $\calC^{(\alpha)} = \calC^{(\alpha)}(0)$, $\calC_n^{(\alpha)} = \calC_n^{(\alpha)}(0)$ and $\calC_{n_1,n_2}^{(\alpha)} = \calC_{n_1,n_2}^{(\alpha)}(0)$.

We will need the following results, which are more specific to cones. As we explain, they can all be easily obtained from standard results and reasonings.

\begin{enumerate}[(i)] \setcounter{enumi}{3}
\item \emph{Near-critical stability for one-arm events in a cone.} For all $\alpha \in (0,\frac{\pi}{2}]$, $K \geq 1$, there exist $C_1, C_2 > 0$ (depending only on $\alpha$ and $K$) such that: for all $p \in (0,1)$, $0 \leq n_1 < n_2 \leq K L(p)$,
\begin{equation} \label{eq:near_critical_arm}
c_1 \pi_1^{(\alpha)}(n_1, n_2) \leq \PP_p \big( \calA_1(\calC_{n_1,n_2}^{(\alpha)}) \big) \leq c_2 \pi_1^{(\alpha)}(n_1, n_2).
\end{equation}
This result can be obtained through similar reasonings as Theorem~27 in \cite{No2008} (which regards arm events in the full plane).

\item \emph{Estimates on $\pi_1^{(\alpha)}$.} For all $\alpha \in (0,\frac{\pi}{2}]$, there exist $c_1, c_2 > 0$ (depending only on $\alpha$) such that: for all $0 \leq n_1 < n_2$,
\begin{equation} \label{eq:one_arm_ext}
\pi_1^{(\alpha)} \bigg( \frac{n_1}{2}, n_2 \bigg), \: \pi_1^{(\alpha)}(n_1, 2 n_2) \geq c_1 \pi_1^{(\alpha)}(n_1, n_2),
\end{equation}
and for all $0 \leq n_1 < n_2 < n_3$,
\begin{equation} \label{eq:one_arm_qm}
c_1 \pi_1^{(\alpha)}(n_1, n_3) \leq \pi_1^{(\alpha)}(n_1, n_2) \pi_1^{(\alpha)}(n_2, n_3) \leq c_2 \pi_1^{(\alpha)}(n_1, n_3).
\end{equation}
The first property shows the ``extendability'' of the arm in a cone, and the second property is usually called quasi-multiplicativity. Even though substantial work is required to establish them in the case of general polychromatic arm events in the plane (see, resp., Propositions~16 and 17 in \cite{No2008}), in this particular situation of a single occupied arm, they are almost direct consequences of the Russo-Seymour-Welsh bounds at criticality, used in combination with the Harris inequality.

\item \emph{One-arm exponent in a cone.} For all $\alpha \in (0,\frac{\pi}{2}]$, let
\begin{equation} \label{eq:one_arm_cone_exp}
\alpha_1^{(\alpha)} = \frac{\pi}{2 \alpha} \cdot \frac{1}{3}.
\end{equation}
Then, for all $\ve > 0$, there exist $c_i(\alpha,\ve) > 0$, $i=1,2$, such that: for all $1 \leq n_1 < n_2$,
\begin{equation} \label{eq:one_arm_cone_prob}
c_1 \bigg( \frac{n_1}{n_2} \bigg)^{\alpha_1^{(\alpha)} + \ve} \leq \pi_1^{(\alpha)}(n_1,n_2) \leq c_1 \bigg( \frac{n_1}{n_2} \bigg)^{\alpha_1^{(\alpha)} - \ve}.
\end{equation}
This exponent can be obtained from the conformal invariance property of critical percolation in the scaling limit \cite{Sm2001}.
\end{enumerate}

The following a-priori estimate will be useful.

\begin{lemma} \label{lem:apriori_sum}
For all $\alpha > \frac{\pi}{6}$, we have: for all $n \geq 1$,
\begin{equation} \label{eq:apriori_sum}
\sum_{k=1}^n \big( \pi_1^{(\alpha)}(k,n) \big)^{-1} \leq C n,
\end{equation}
where $C$ depends only on $\alpha$.
\end{lemma}

\begin{proof}
Let $\alpha > \frac{\pi}{6}$, and consider the corresponding exponent $\alpha_1^{(\alpha)}$, which we know is $< 1$ (from \eqref{eq:one_arm_cone_exp}). Hence, we can let $\ve > 0$ so that $\alpha_1^{(\alpha)} + 2 \ve = 1$. From \eqref{eq:one_arm_cone_prob}, we have
$$\pi_1^{(\alpha)}(k,n) \geq c_1 \bigg( \frac{k}{n} \bigg)^{\alpha_1^{(\alpha)} + \ve} = c_1 \bigg( \frac{k}{n} \bigg)^{1 - \ve}$$
for some $c_1(\alpha) > 0$. We deduce immediately
\begin{equation}
\sum_{k=1}^n \big( \pi_1^{(\alpha)}(k,n) \big)^{-1} \leq (c_1)^{-1} n^{1 - \ve} \cdot \sum_{k=1}^n k^{-(1 - \ve)} \leq (c_1)^{-1} n^{1 - \ve} \cdot c'_1 n^{\ve} = C n,
\end{equation}
which gives \eqref{eq:apriori_sum}, and completes the proof.
\end{proof}

\section{Cone sites} \label{sec:cone_sites}

In this section, we develop the main geometric idea used in our proofs: we move away from the boundary by considering paths contained in cones, included in $\HH$ and with an opening angle just slightly below $\pi$.

In the following, we consider the pure birth process in $\HH$ (with rate $1$), i.e. Bernoulli percolation with parameter $p(t) = 1 - e^{-t}$. Moreover, we fix some $\zeta > 0$, and we ``trigger'' the vertices along $\dout V_{\HH}$, with $y$-coordinate $- \frac{\sqrt{3}}{2}$, according to a Poisson process with rate $\zeta$. Each time such a vertex is triggered, we consider the two vertices in $V_{\HH}$ above it: for each of them, if it is occupied, we put a mark on all occupied vertices connected to it at this time (and otherwise, if it is vacant, nothing happens). For each $t \geq 0$, we then denote by $\calF_t$ the set of all vertices carrying a mark at time $t$, which was thus transmitted by a vertex triggered at an earlier time.

We start by the following observation. For $v \in \din V_{\HH}$ and $n \geq 1$, we denote $F_n(v) := \{$one of the two neighbors of $v$ on $\dout V_{\HH}$ gets triggered at some time $t \in [0,t_c]$, at which $\calA_1(\calC_n^{(\frac{\pi}{2})}(v))$ occurs$\}$. In other words, $F_n(v)$ is the event that in the pure birth process, there exists an occupied path (in $V_{\HH}$), which is ignited by a neighbor of $v$ before time $t_c$ and reaches distance $n$ from $v$. We have the following estimate.

\begin{lemma} \label{lem:long_path}
Let $\zeta \in (0,\infty)$, and $\ve > 0$. There exists $c(\ve)$ such that for all $v \in \din V_{\HH}$ and $n \geq 1$,
\begin{equation} \label{eq:long_path}
\PP \big( F_n(v) \big) \leq c \zeta n^{-\frac{13}{12} + \ve}.
\end{equation}
\end{lemma}

Establishing this upper bound \eqref{eq:long_path} requires a little bit of care, and we do it below. However, note that it is easy to explain heuristically the exponent $\frac{13}{12}$ that appears, as follows. For a connection to distance $n$ to have a reasonable probability to form, the parameter $p(t) \leq p_c$ needs to be such that $L(p(t)) \gtrsim n$. In other words, the ignition has to take place in the corresponding near-critical window, which has length $n^{-\frac{3}{4} + o(1)}$ (using the critical exponent $\frac{4}{3}$ for $L$, see \eqref{eq:theta_L_exp}). In this case, there exists an occupied path to distance $n$ with a probability $n^{-\frac{1}{3} + o(1)}$, from \eqref{eq:one_arm_cone_exp} and \eqref{eq:one_arm_cone_prob}. Hence, we obtain $\approx \zeta n^{-\frac{3}{4} + o(1)} \cdot n^{-\frac{1}{3} + o(1)} = \zeta n^{-\frac{13}{12} + o(1)}$.

Let us now prove this estimate formally.

\begin{proof}
For notational convenience, we write $\pi_1^+ = \pi_1^{(\frac{\pi}{2})}$ in this proof (only). Note that \eqref{eq:one_arm_cone_exp} provides the value of the corresponding exponent:
\begin{equation} \label{eq:exp_hp}
\alpha_1^+ = \alpha_1^{(\frac{\pi}{2})} = \frac{1}{3}.
\end{equation}
The lemma follows from a summation argument similar to those in the proofs of Lemma~6.8 in \cite{BN2021} and Lemma~3.3 in \cite{LN2021}. Let $\delta > 0$ be such that $L(t_c - \delta) = n$. Without loss of generality, we can assume that $n$ is large enough so that $t_c - \delta > \frac{3}{4} t_c$, and introduce the integer $J \geq 0$ satisfying
$$t_c - 2^{J+1} \delta \leq \frac{3}{4} t_c < t_c - 2^J \delta.$$
Observe that in particular, we have necessarily $t_c - 2^{J+1} \delta > \frac{1}{2} t_c$.

We then bound the desired probability from above by summing according to the subinterval $[t_c - 2^{j+1} \delta, t_c - 2^j \delta)$, $0 \leq j \leq J$, containing the time $t$ at which one of the neighbors of $v$ gets triggered (note that there might be several such times, but this is not an issue). If $t \in [t_c-\delta, t_c]$, we have $L(t) \geq n$, and we can simply use $2 \zeta \delta \pi_1^+(n)$ as an upper bound. If $t \in (0, t_c - 2^{J+1} \delta)$, we use the upper bound $2 \zeta t_c c_1 e^{-c_2 n/L(\frac{3}{4} t_c)}$ (coming from \eqref{eq:exp_decay_L1}). Hence, we obtain from the union bound:
\begin{equation} \label{eq:sum_pf1}
\PP \big( F_n(v) \big) \leq 2 \zeta \bigg( \delta \pi_1^+(n) + \sum_{j=0}^J 2^j \delta \pi_1^+ \big( L(t_c - 2^{j+1} \delta) \big) 4 c_1 e^{-c_2 n/L(t_c - 2^j \delta)} + t_c c_1 e^{-c_2 n/L(\frac{3}{4} t_c)} \bigg).
\end{equation}
We will use that the following bounds hold for $L(t_c - 2^j \delta)$: there exist constants $c'_1, c'_2 > 0$ (depending only on $\ve$) such that for all $j \in \{0, \ldots, J\}$,
\begin{equation} \label{eq:ratio_L}
c'_1 (2^j)^{-\frac{4}{3} - \ve} n \leq L(t_c - 2^j \delta) \leq c'_2 (2^j)^{-\frac{4}{3} + \ve} n.
\end{equation}
Indeed, this follows by writing
$$L(t_c - 2^j \delta) = \frac{L(t_c - 2^j \delta)}{L(t_c - \delta)} L(t_c - \delta) = \frac{L(t_c - 2^j \delta)}{L(t_c - \delta)} n,$$
and estimating the ratio above by using \eqref{eq:L_equiv}, combined with classical bounds on the four-arm (full-plane) probability $\pi_4(n_1,n_2)$ (see e.g. Lemma~2.5 in \cite{BN2021}, and also the quasi-multiplicativity property for $\pi_4$, which is (2.6) in that paper).

Consider some $j \in \{0,\ldots,J\}$. On the one hand,
\begin{equation} \label{eq:sum_pf2}
\pi_1^+ \big( L(t_c - 2^{j+1} \delta) \big) \leq c_3 \pi_1^+ \big( L(t_c - \delta) \big) \pi_1^+ \big( L(t_c - 2^{j+1} \delta), L(t_c - \delta) \big)^{-1} \leq c'_3 (2^{j+1})^{\frac{4}{3} \cdot \frac{1}{3} + \ve} \pi_1^+ (n),
\end{equation}
using \eqref{eq:one_arm_qm}, and then \eqref{eq:one_arm_cone_prob} (combined with the value from \eqref{eq:exp_hp}, as well as \eqref{eq:ratio_L}). Here, $c_3$ is universal, and $c'_3$ depends only on $\ve$. On the other hand,
\begin{equation} \label{eq:sum_pf3}
e^{-c_2 n/L(t_c - 2^j \delta)} \leq e^{-c''_2 (2^j)^{\frac{4}{3} - \ve}}.
\end{equation}
(from \eqref{eq:ratio_L}), for some $c''_2(\ve) > 0$. By combining \eqref{eq:sum_pf1}, \eqref{eq:sum_pf2} and \eqref{eq:sum_pf3}, we obtain
$$\PP \big( F_n(v) \big) \leq 2 \zeta \bigg( \delta \pi_1^+(n) + \delta \pi_1^+(n) \sum_{j \geq 0} c'_3 (2^{j+1})^{\frac{4}{9} + \ve} e^{-c''_2 (2^j)^{\frac{4}{3} - \ve}} + t_c c_1 e^{-c_2 n/L(\frac{3}{4} t_c)} \bigg).$$
This allows us to conclude, using finally that for some constants $c_4, c'_4, c_5$ depending only on $\ve$, $n = L(t_c - \delta) \leq c_4 \delta^{-\frac{4}{3}/(1-\frac{\ve}{2})}$, so $\delta \leq c'_4 n^{-\frac{3}{4} + \frac{\ve}{2}}$, and $\pi_1^+(n) \leq c_5 n^{- \frac{1}{3} + \frac{\ve}{2}}$ (from \eqref{eq:one_arm_cone_prob} and \eqref{eq:exp_hp}):
\begin{equation} \label{eq:Fn_upper_bd}
\PP \big( F_n(v) \big) \leq c \zeta n^{-\frac{3}{4} + \frac{\ve}{2}} n^{- \frac{1}{3} + \frac{\ve}{2}} = c \zeta n^{-\frac{13}{12} + \ve}.
\end{equation}
We have thus established \eqref{eq:long_path}, which completes the proof.
\end{proof}

We are now ready to introduce cone sites. We adopt the definition below, illustrated on Figure~\ref{fig:cone_site}.

\begin{definition} \label{def:cone_site}
Let $\zeta \in (0,\infty)$. Let $\alpha > \frac{\pi}{6}$, and $n > 0$. A vertex $v \in \din V_{\HH}$ is called an \emph{$(\alpha,n)$-cone site} if the two conditions below are satisfied:
\begin{enumerate}[(i)]
\item $\calF_{t_c} \cap \calC^{(\alpha)}(v) = \emptyset$,

\item and $\calA_1(\calC_n^{(\alpha)}(v))$ occurs at $t_c$, i.e. there exists (in the pure birth process) an occupied arm connecting $v$ to distance $n$ within the cone $\calC^{(\alpha)}(v)$.
\end{enumerate}
Note that in particular, $v$ has to be occupied.
\end{definition}

\begin{figure}[t]
\begin{center}

\includegraphics[width=.85\textwidth]{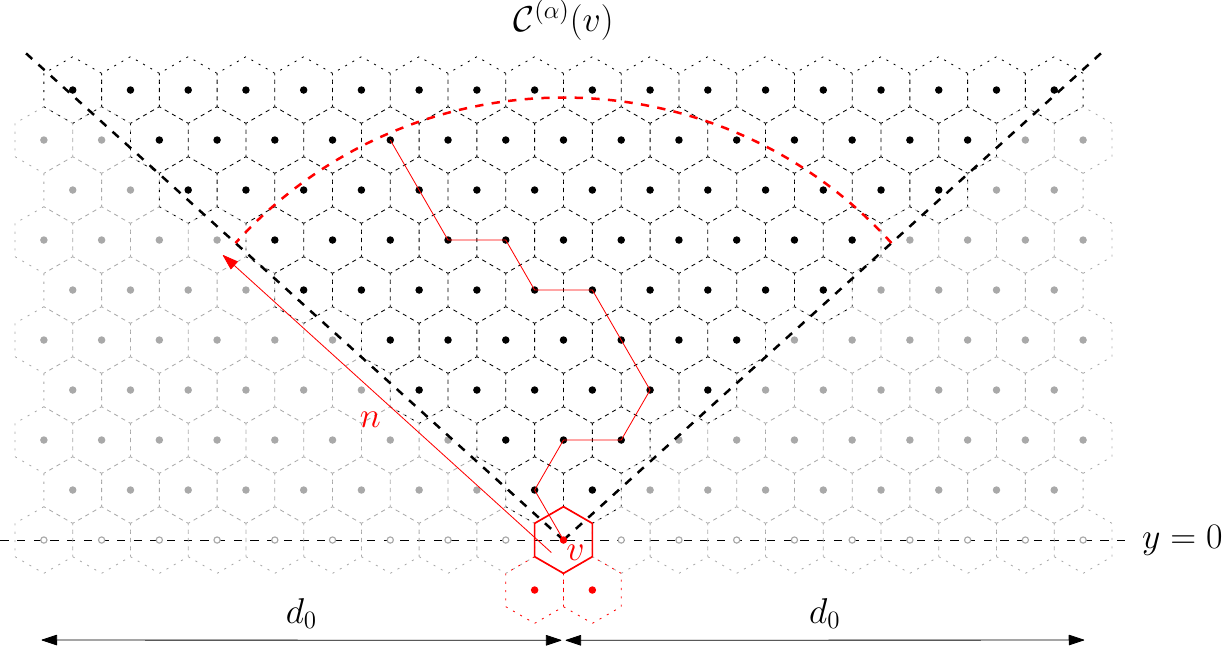}
\caption{\label{fig:cone_site} For some $\alpha \in (\frac{\pi}{6}, \frac{\pi}{2})$ and $v \in \din V_{\HH}$, the cone $\calC^{(\alpha)}(v)$ is shown in black. This figure depicts the conditions for a vertex $v \in \din V_{\HH}$, i.e. with $y$-coordinate equal to $0$, to be an $(\alpha,n)$-cone site, $n \geq 0$: (i) There is no occupied path which is ignited before time $t_c$ and enters the cone; (ii) There exists a $t_c$-occupied path connecting $v$ to distance $n$ within $\calC^{(\alpha)}(v)$. The two neighboring vertices below $v$, colored in red, can then be used as ``triggers''. We prove that the probability for $v$ to be a cone site is comparable to $\pi_1^{(\alpha)}(n) = \PP_{\frac{1}{2}}( \calA_1(\calC_n^{(\alpha)}) )$: the upper bound is clear, and the lower bound is obtained by making the further requirement that all the vertices along $\din V_{\HH}$ at a distance at most $d_0$ from $v$ (marked with white disks) are $2t_c$-vacant, for some well-chosen $d_0(\alpha)$. In terms of probability, this extra condition has a constant cost, once $d_0$ is fixed.}

\end{center}
\end{figure}

Later in the paper, we use this notion twice, or rather small variations of it, in two different situations. First, in Section~\ref{sec:ff_hexagon}, cone sites are used to study the spread of ignitions in large finite domains. And then in Section~\ref{sec:connection_Graf}, we explain how cone sites can be used to gain further insight on Graf's forest fire process in the upper half-plane. We will make use of the following estimate.

\begin{lemma} \label{lem:cone_est}
Let $\zeta \in (0,\infty)$. For any $\alpha \geq \frac{\pi}{6}$, we have the following estimates.
\begin{enumerate}[(i)]
\item There exists $c_1(\alpha,\zeta) > 0$ such that: for all $v \in \din V_{\HH}$ and $n \geq 0$,
\begin{equation} \label{eq:cone_est1}
c_1 \pi_1^{(\alpha)}(n) \leq \PP \big( v \text{ is an $(\alpha,n)$-cone site} \big) \leq \pi_1^{(\alpha)}(n).
\end{equation}

\item There exists $c_2(\alpha) > 0$ such that: for all $v, v' \in \din V_{\HH}$ and $n \geq 0$,
\begin{equation} \label{eq:cone_est2}
\PP \big( v \text{ and } v' \text{ are both $(\alpha,n)$-cone sites} \big) \leq c_2 \pi_1^{(\alpha)}(n) \cdot \pi_1^{(\alpha)}(\|v-v'\| \wedge n).
\end{equation}
\end{enumerate}
\end{lemma}

\begin{proof}
We first consider (i). The second inequality in \eqref{eq:cone_est1} is clear from the definition, so we only need to prove the first one. For this purpose, consider, for some given $d_0 \geq 1$ that we explain how to choose later, the additional event that all vertices $v' \in \din V_{\HH}$ with $\|v-v'\| \leq d_0$ are $2t_c$-vacant. This has a fixed cost $(1 - p(2t_c))^{2 d_0} = e^{-2 d_0 t_c}$, and under this condition, marks inside $\calC^{(\alpha)}(v)$ can only be created by triggered vertices at a distance $\geq d_0 + 1$ from $v$.

Now, let $v' \in \din V_{\HH}$ with $k:= \|v-v'\| \geq d_0+1$. We observe that a path from $v'$ to $\calC^{(\alpha)}(v)$ implies in particular the existence of an occupied path from $v'$ to distance $d(k) = k \cos \alpha$, so Lemma~\ref{lem:long_path} implies the following. For any $\ve > 0$, the probability that one of the two neighbors of $v'$ gets triggered at some time $t \in [0,t_c]$ at which there exists an occupied path from $v'$ to the cone is at most $\zeta c d^{-\frac{13}{12} + \ve}$, for some $c(\ve)$. Choosing $\ve = \frac{1}{24}$, we deduce from the union bound that the probability that such a $v'$ exists is at most
$$2 \zeta \sum_{k=d_0}^{\infty} c (k \cos \alpha)^{- \frac{25}{24}} \leq c' (d_0)^{- \frac{1}{24}}.$$
In particular, it can be made $\leq \frac{1}{2}$ by choosing $d_0$ sufficiently large (in terms of $\alpha$ and $\zeta$ only), which we do. We finally obtain that
$$\PP(v \text{ is an $(\alpha,n)$-cone site}) \geq \frac{1}{2} e^{-2 d_0 t_c} \pi_1^{(\alpha)}(n),$$
which completes the proof of (i).

Let us now turn to (ii), and let $d = \|v-v'\|$. If $d \geq \frac{n}{2}$, the event that $v$ and $v'$ are both cone sites implies in particular the existence of occupied arms from $v$ and $v'$, both to distance $\frac{n}{4}$. We deduce that
$$\PP(v \text{ and } v' \text{ are both $(\alpha,n)$-cone sites}) \leq \Big( \pi_1^{(\alpha)}\Big( \frac{n}{4} \Big) \Big)^2,$$
and we can conclude in this case by using the extendability property \eqref{eq:one_arm_ext}. Next, we consider $d \leq \frac{n}{2}$. In this case, we make appear an arm in $\calC_{\frac{d}{2}}^{(\alpha)}(v)$ and an arm in $\calC_{\frac{d}{2}}^{(\alpha)}(v')$, and also an arm in $\calC_{2d,n}^{(\alpha)}(v')$.
We obtain
$$\PP(v \text{ and } v' \text{ are both $(\alpha,n)$-cone sites}) \leq \Big( \pi_1^{(\alpha)}\Big( \frac{d}{2} \Big) \Big)^2 \pi_1^{(\alpha)}( 2d, n ) \leq c_2 \pi_1^{(\alpha)}(n) \pi_1^{(\alpha)}(d),$$
where we used \eqref{eq:one_arm_ext} and \eqref{eq:one_arm_qm} in the second inequality. This completes the proof.
\end{proof}

Next, we use a second-moment argument, based on Lemma~\ref{lem:cone_est}, to check that there are typically plenty of cone sites.

\begin{lemma} \label{lem:cone_number}
Let $\alpha \in (\frac{\pi}{6}, \frac{\pi}{2})$, and $\zeta \in (0,\infty)$. For $n \geq 0$ and $\delta > 0$, let
$$V_n = V^{(\alpha),\delta}_n := \big| \big\{ v \in ([-n,n] \times \{0\})\cap V_{\HH} \: :\ v \text{ is an } (\alpha,\delta n)\text{-cone site} \big\} \big|.$$
Then for each $\ve > 0$, there exists $c_1(\alpha,\zeta), c_2(\alpha,\zeta,\ve), c_3(\alpha,\zeta) > 0$ such that for all $n$ large enough,
\begin{equation} \label{eq:cone_number}
\PP \big( V_n \geq c_1 n \pi_1^{(\alpha)}(\delta n) \big) \geq 1 - c_2 n \cdot (\delta n)^{-\frac{13}{12} + \ve} - c_3 \delta.
\end{equation}
\end{lemma}

\begin{proof}
On the one hand, \eqref{eq:cone_est1} implies directly that for some $c = c(\alpha,\zeta) > 0$,
\begin{equation} \label{eq:lower_bd_exp_V}
\EE \big[ V_n \big] \geq c n \pi_1^{(\alpha)}(\delta n).
\end{equation}
In order to use a second-moment reasoning, we replace each event $I_n(v) := \{v$ is an $(\alpha,\delta n)$-cone site$\}$ by a ``localized'' version $\tilde{I}_n(v)$ (depending on $\alpha$ and $\delta$), obtained by considering only paths and ignitions within the box
$$R_n(v) = R_n^{(\alpha),\delta}(v) := \Big( v + [- (\tan \alpha) \delta n, (\tan \alpha) \delta n ] \times \Big[- \frac{\sqrt{3}}{2}, \delta n \Big] \Big) \cap V_{\HH}.$$
More precisely, in the definition that $v$ is an $(\alpha,\delta n)$-cone site, we replace the first condition (that is, (i) in Definition~\ref{def:cone_site}) by
$$\text{(i)'} \:\: \tilde{\calF}_{t_c}(v) \cap \calC^{(\alpha)}(v) = \emptyset,$$
where $\tilde{\calF}_{t_c}(v)$ is the set of vertices that can be reached by a local path before time $t_c$: i.e., a path which is marked by the triggering of a vertex $v'$ with $\|v-v'\| \leq (\tan \alpha) \delta n$, and which stays completely inside $R_n(v)$ (see Figure~\ref{fig:local_event}).

\begin{figure}[t]
\begin{center}

\includegraphics[width=.48\textwidth]{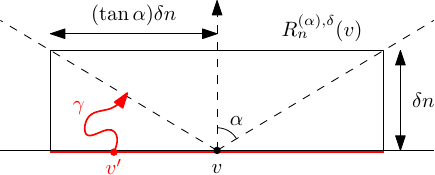}
\caption{\label{fig:local_event} For the localized event $\tilde{I}_n(v)$, we only take into account ignitions from the boundary arc $(v + [- (\tan \alpha) \delta n, (\tan \alpha) \delta n] \times \{ - \frac{\sqrt{3}}{2} \} ) \cap V_{\HH}$ (colored in red), and ignited paths $\gamma$ remaining in $R_n^{(\alpha),\delta}(v)$.}

\end{center}
\end{figure}

We denote by $\tilde{V}_n$ the corresponding number of vertices $v$. Clearly, $I_n(v) \subseteq \tilde{I}_n(v)$, so $\tilde{V}_n \geq V_n$. On the other hand, we observe that the event $\tilde{I}_n(v) \setminus I_n(v)$ implies the existence of an ignited path (marked by the triggering of a vertex) with radius $\geq \kappa \delta n$, for some suitable $\kappa(\alpha) > 0$. Indeed, assume that $\tilde{I}_n(v)$ occurs, but not $I_n(v)$: this means that a path $\gamma$, ignited by some vertex $v'$, reaches the cone $\calC^{(\alpha)}(v)$ (before time $t_c$), but there exists no such local path. We can then distinguish two cases. If $\|v-v'\| \geq (\tan \alpha) \delta n/2$, then $\gamma$ has radius at least $(\sin \alpha) \delta n /2$. On the other hand, if $\|v-v'\| \leq (\tan \alpha) \delta n/2$, then $\gamma$ has to exit the box $R_n(v)$ (otherwise it would be a local path), so it connects $v'$ to a distance at least $(\tan \alpha) \delta n/2$. This establishes the claim.

This leads us to consider the event
$$\calJ_n := \Big\{\exists v' \in \Big([-2n,2n] \times \Big\{- \frac{\sqrt{3}}{2}\Big\}\Big)\cap V_{\HH} \text{ triggering an ignited path with radius $\geq \kappa \delta n$} \Big\},$$
and we let $\calJ'_n$ be the event that some ignition from outside $[-2n,2n] \times \{- \frac{\sqrt{3}}{2} \}$ reaches one of the cones $\calC^{(\alpha)}(v)$, $v \in ([-n,n] \times \{0\})\cap V_{\HH}$. Hence, $V_n = \tilde{V}_n$ on the event $(\calJ_n \cup \calJ'_n)^c$. Using Lemma~\ref{lem:long_path}, we can easily obtain that for some $c_2(\alpha,\zeta,\ve)$,
\begin{equation} \label{eq:upper_bd_Jn}
\PP(V_n \neq \tilde{V}_n) \leq \PP(\calJ_n \cup \calJ'_n) \leq c_2 n \cdot (\delta n)^{-\frac{13}{12} + \ve}.
\end{equation}
Now, there remains to estimate $\text{Var}(\tilde{V}_n)$. Noticing that for any two vertices $v, v' \in ([-n,n] \times \{- \frac{\sqrt{3}}{2}\})\cap V_{\HH}$, the events $\tilde{I}_n(v)$ and  $\tilde{I}_n(v')$ are independent as soon as $\|v - v'\| \geq 2 (\tan \alpha) \delta n$, we obtain
\begin{align*}
\text{Var}(\tilde{V}_n) & \leq \sum_{v \in [-n,n] \times \{0\}} \sum_{v' : \|v-v'\| \leq 2 (\tan \alpha) \delta n} \PP \big( \tilde{I}_n(v) \cap \tilde{I}_n(v') \big)\\
& \leq \sum_{v \in [-n,n] \times \{0\}} 2 \sum_{d=0}^{2 (\tan \alpha) \delta n} c^{(1)} \pi_1^{(\alpha)}(\delta n) \cdot \pi_1^{(\alpha)}(d \wedge \delta n)\\
& = \sum_{v \in [-n,n] \times \{0\}} 2 \bigg( \sum_{d=0}^{\delta n} c^{(1)} \pi_1^{(\alpha)}(\delta n) \cdot \pi_1^{(\alpha)}(d) + \sum_{d=\delta n+1}^{2 (\tan \alpha) \delta n} c^{(1)} \pi_1^{(\alpha)}(\delta n) \cdot \pi_1^{(\alpha)}(\delta n) \bigg)\\
& \leq \sum_{v \in [-n,n] \times \{0\}} 2 \bigg( \sum_{d=0}^{\delta n} c^{(2)} \pi_1^{(\alpha)}(\delta n)^2 \cdot \big( \pi_1^{(\alpha)}(d,\delta n) \big)^{-1} + c^{(3)} \delta n \pi_1^{(\alpha)}(\delta n)^2 \bigg).
\end{align*}
where we used \eqref{eq:cone_est2} for the second inequality, and the quasi-multiplicativity property \eqref{eq:one_arm_qm} on the last line. In this series of inequalities, each of the constants $c^{(i)}$ only depends on $\alpha$. From Lemma~\ref{lem:apriori_sum}, we have
$$\sum_{d=0}^{\delta n} \big( \pi_1^{(\alpha)}(d,\delta n) \big)^{-1} \leq c^{(4)} \delta n$$
(here we use the hypothesis $\alpha > \frac{\pi}{6}$), so
$$\text{Var}(\tilde{V}_n) \leq 2n \pi_1^{(\alpha)}(\delta n)^2 \cdot (c^{(5)} \delta n) = c^{(6)} \delta (n \pi_1^{(\alpha)}(\delta n))^2.$$
It thus follows from Chebyshev's inequality, combined with \eqref{eq:lower_bd_exp_V} and $\tilde{V}_n \geq V_n$, that
\begin{equation}
\PP \Big( \tilde{V}_n \geq \frac{c}{2} n \pi_1^{(\alpha)}(\delta n) \Big) \geq 1 - \frac{4 \text{Var}(\tilde{V}_n)}{(c n \pi_1^{(\alpha)}(\delta n))^2} \geq 1 - c_3 \delta,
\end{equation}
for some $c_3(\alpha,\zeta)$. If we let $c_1 = \frac{c}{2}$, we obtain, thanks also to \eqref{eq:upper_bd_Jn},
$$\PP \big( V_n \geq c_1 n \pi_1^{(\alpha)}(\delta n) \big) \geq 1 - c_2 n \cdot (\delta n)^{-\frac{13}{12} + \ve} - c_3 \delta,$$
which completes the proof.
\end{proof}

\section{Forest fire with boundary ignitions} \label{sec:ff_hexagon}

We now investigate the behavior in finite domains of the forest fire model with boundary ignitions. For convenience, we focus on hexagonal domains fitting on the triangular lattice. First, we set notations in Section~\ref{sec:ff_notation}, defining precisely the processes under consideration. We then prove our main results, Theorems~\ref{thm:main_result1} and \ref{thm:main_result2}, for the processes without and with recoveries, respectively, in Sections~\ref{sec:proof} and \ref{sec:proof2}.

\subsection{Setting and notations} \label{sec:ff_notation}

We now define precisely the forest fire model that we study, and set notations. Our process is defined on vertices in the hexagon $H_N$ centered on $0$ and with side length $N$, which is depicted on Figure~\ref{fig:hexagon}. Formally, $H_N$ is the set of vertices at a graph distance at most $N$ from $0$, i.e. which can be reached from $0$ through a path of length at most $N$. Vertices along $\dout H_N$ (which consists of the vertices at distance exactly $N+1$ from $0$) get ignited at some given rate $\zeta \in (0,\infty]$, and trigger ignitions within the hexagon: when such a vertex gets triggered, all the occupied vertices connected to it burn immediately (similarly to the half-plane setting in Section~\ref{sec:cone_sites}). We denote by $\PP_N$ the probability measure governing this process. In our proofs, we mostly use the the lower side of $\din H_N$, and the ignitions produced by the vertices on the row just below, i.e. with $y$-coordinate $- (N+1) \frac{\sqrt{3}}{2}$. For this purpose, we will naturally shift the definition of cone sites vertically, by $- N \frac{\sqrt{3}}{2}$.

\begin{figure}[t]
\begin{center}

\includegraphics[width=.85\textwidth]{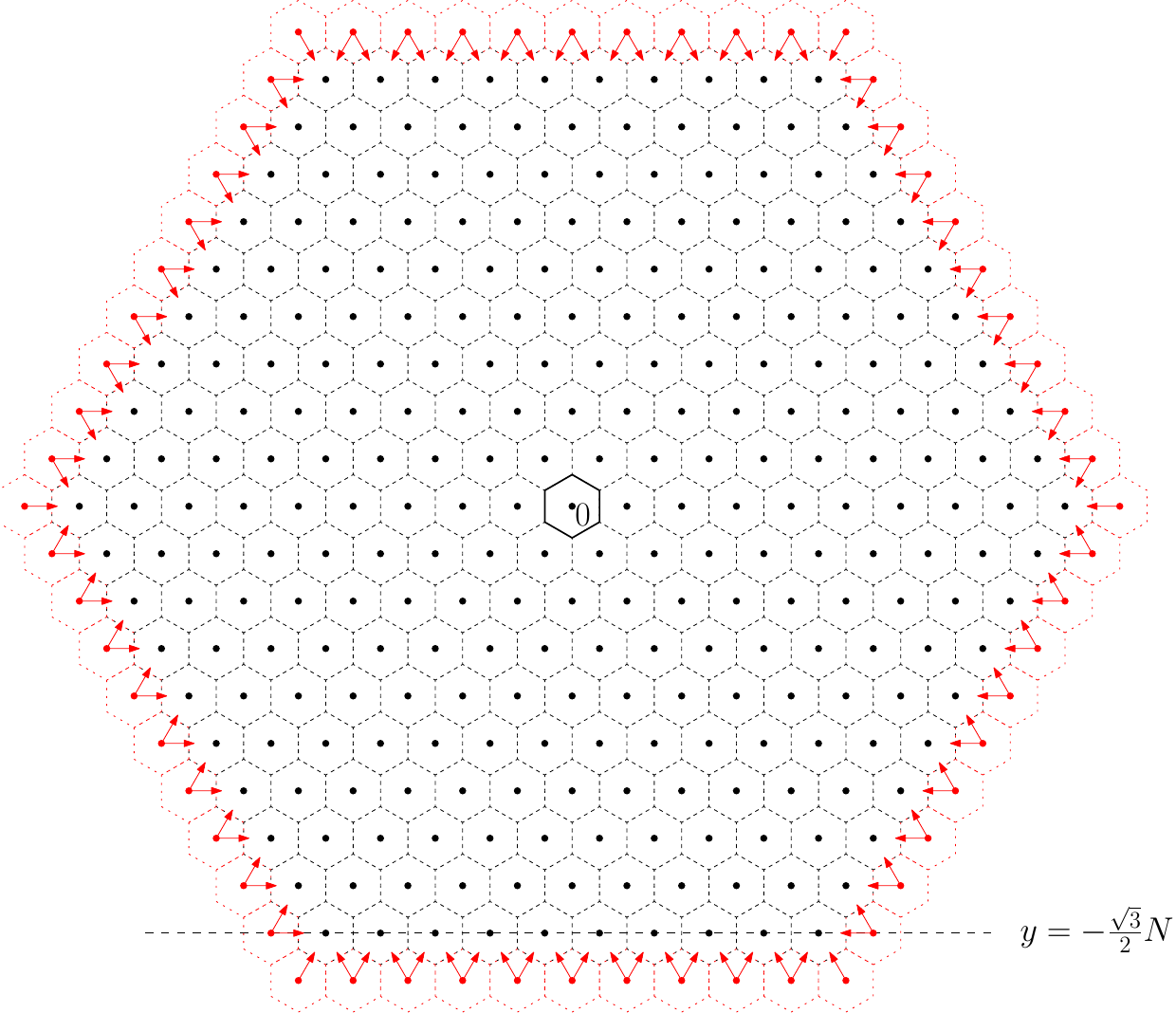}
\caption{\label{fig:hexagon} The hexagon $H_N$ consists of all the vertices lying within a graph distance $N$ from $0$ (colored in black). The ignitions are then coming from the vertices along its outer boundary $\dout H_N$, shown in red.}

\end{center}
\end{figure}

Once again, for each $t \in [0, \infty)$, we denote by $\calF_t$ the set of vertices which are carrying a mark at time $t$ (i.e. which were reached by a marked path before time $t$, produced by the triggering of some vertex during $[0,t]$). The following observation will turn out to be handy.

\begin{lemma} \label{lem:depth_fires}
Let $\zeta \in (0,\infty)$. For all $\delta \in (0, \frac{1}{13})$, we have:
\begin{equation} \label{eq:depth_fires}
\text{for all } \beta > \frac{3}{4} (1- \delta), \quad \PP \big( \calF_{t_c + N^{- \beta}} \cap H_{N-N^{1-\delta}} \neq \emptyset \big) \stackrel{N \to \infty}{\longrightarrow} 0.
\end{equation}
\end{lemma}
In the following, we use this lemma through the fact that \eqref{eq:depth_fires} holds with some $\delta > 0$ and some $\beta$ slightly smaller than $\frac{3}{4}$, so that the time $t_c + N^{-\beta}$ is ``much later'' than $t_c + N^{-3/4}$. We can for example set $\delta_0 = \frac{1}{14}$, and $\beta_0 = \frac{7}{10}$. We will see that the requirement $\delta < \frac{1}{13}$ appears naturally in the proof.

\begin{proof}
Let $\delta > 0$. A similar summation as that in the proof of Lemma~\ref{lem:long_path} gives the following analog of \eqref{eq:Fn_upper_bd}. For each $\ve > 0$, the probability, for a marked path originating from $v \in \dout H_N$, to reach distance $N^{1 - \delta}$ before time $t_c + N^{-\beta}$ is at most $(N^{1 - \delta})^{-\frac{1}{3} + \ve} (N^{1 - \delta})^{-\frac{3}{4} + \ve}$, for all $N$ large enough. Here we also use the condition $\beta > \frac{3}{4} (1- \delta)$, which ensures that $L(t_c + N^{- \beta}) = N^{\frac{4}{3} \beta + o(1)} \gg N^{1-\delta}$. From the union bound, we conclude that
\begin{equation} \label{eq:depth_union_bd}
\PP \big( \calF_{t_c + N^{- \beta}} \cap H_{N-N^{1-\delta}} \neq \emptyset \big) \leq c N \cdot (N^{1-\delta})^{-\frac{13}{12} + 2 \ve}.
\end{equation}
Now, if we assume that $\delta < \frac{1}{13}$, and we then choose $\ve > 0$ small enough, the right-hand side of \eqref{eq:depth_union_bd} tends to $0$, which completes the proof.
\end{proof}

%

\subsection{Proof of Theorem~\ref{thm:main_result1}} \label{sec:proof}

We are now in a position to prove our main result for the forest fire process in $H_N$.

\begin{proof}[Proof of Theorem~\ref{thm:main_result1}]
We consider first $\zeta \in (0,\infty)$. As we explain toward the end of the proof, the case $\zeta = \infty$ can be handled in the same way, with only a small adaptation in the definition of cone sites being required.

For any given $\bar{\ve} > 0$, let us prove that for all $N$ large enough, $\PP_N(0 \text{ gets ignited}) \leq \bar{\ve}$. Consider some arbitrary $\delta \in (0, \frac{1}{13})$ and $\beta \in (\frac{3}{4} (1- \delta), \frac{3}{4})$. Let $\ul t := t_c + N^{- \beta}$. We know from Lemma~\ref{lem:depth_fires} that
\begin{equation} \label{eq:proof_main1}
\text{for all $N \geq N_1$,} \quad \PP \big( \calF_{\ul t} \cap H_{N-N^{1-\delta}} = \emptyset \big) \geq 1 - \frac{\bar{\ve}}{5}.
\end{equation}
Let $\beta' \in (\beta, \frac{3}{4})$. Using the critical exponent for $L$ (see \eqref{eq:theta_L_exp}), we have $L(\ul t) = N^{\frac{4}{3} \beta + o(1)} \ll N^{\frac{4}{3} \beta'}$ as $N \to \infty$, so we deduce from \eqref{eq:exp_decay_L2} that (in the pure birth process):
\begin{equation} \label{eq:proof_main2}
\text{for all $N \geq N_2$,} \quad \PP \big( \text{at time $\ul t$, } \calO(H_{N-N^{1-\delta}} \setminus H_{N - N^{4 \beta'/3}}) \text{ occurs} \big) \geq 1- \frac{\bar{\ve}}{5}.
\end{equation}
From now on, we assume that the two events appearing in \eqref{eq:proof_main1} and \eqref{eq:proof_main2} occur, and we denote by $\calO$ any occupied circuit as in \eqref{eq:proof_main2}.

Now, let $\alpha = \frac{\pi}{2} / (1 + \delta)$, so that $\alpha_1^{(\alpha)} = \frac{1}{3} (1+\delta)$ (see \eqref{eq:one_arm_cone_exp}). We will consider cone sites (as always, at time $t_c$) along the lower side $\dout_B H_N$ of $H_N$, to distance $\eta N$, for some well-chosen $\eta>0$. For this, we use a small adaptation of Lemma~\ref{lem:cone_number} (with e.g. the particular value $\ve = \frac{1}{24}$). To be safe, we lower bound the number of cone sites by restricting to vertices along the middle ``third'' of $\dout_B H_N$, that we denote by $\dout_{[B]} H_N$. Even though we are not exactly in the upper half-plane any more, it is easy to see that the same conclusions hold for the number of cone vertices, by truncating some of the summations in the proofs if necessary. Hence, we get that for some $\eta > 0$ small enough:
\begin{equation} \label{eq:proof_main3}
\text{for all $N \geq N_3$,} \quad \PP \big( \text{there exist at least $N^{\frac{2}{3} - \delta}$ $(\alpha,\eta N)$-cone sites along $\dout_{[B]} H_N$} \big) \geq 1- \frac{\bar{\ve}}{5}.
\end{equation}

We now assume that the event in \eqref{eq:proof_main3} holds, in addition to those in \eqref{eq:proof_main1} and \eqref{eq:proof_main2}, and we investigate what happens after time $\ul t$, conditionally on these events. We make the following observations.
\begin{enumerate}[(1)]
\item At time $\ul t$, none of these cone sites has been triggered yet. Indeed, this would otherwise allow $\calF_{t_c + N^{- \beta}}$ to enter into $H_{N-N^{1-\delta}}$, contradicting the event in \eqref{eq:proof_main1}. Moreover, at this time $\ul t$, all the cone sites are connected to the circuit $\calO$.

\item Let $\ol t := t_c + N^{-\frac{2}{3} + 2 \delta} > \ul t$, we have
\begin{equation} \label{eq:proof_main4}
\text{for all $N \geq N_4$,} \quad \PP \big( \text{one of the cone sites gets triggered before time $\ol t$} \big) \geq 1- \frac{\bar{\ve}}{5}.
\end{equation}
When this happens, this causes the circuit $\calO$ to burn, thanks to the previous observation.

\item Finally, $L(\ol t) = N^{\frac{8}{9} - \frac{8}{3} \delta + o(1)}$ as $N \to \infty$, so
\begin{equation} \label{eq:proof_main5}
\text{for all $N \geq N_5$,} \quad \PP \big( \text{at time $\ol t$, $\calO^*(0)$ occurs} \big) \geq 1- \pi_1(N^{\frac{8}{9} - 3 \delta}) \geq 1- \frac{\bar{\ve}}{5},
\end{equation}
where we denote by $\calO^*(0)$ the existence, in the pure birth process, of a vacant circuit which surrounds $0$, and furthermore separates $0$ from $\dout H_{N/2}$.
\end{enumerate}
We can thus conclude, by combining \eqref{eq:proof_main1}, \eqref{eq:proof_main2}, \eqref{eq:proof_main3}, \eqref{eq:proof_main4} and \eqref{eq:proof_main5}, that
$$\text{for all $N \geq N_0 := \max_{1 \leq i \leq 5} N_i$,} \quad \PP_N(\text{$0$ does not get ignited}) \geq 1- \bar{\ve},$$
as desired. This establishes \eqref{eq:main}.

In order to obtain \eqref{eq:main_quant}, i.e. explicit lower and upper bounds on the probability that $0$ gets ignited, it suffices to be slightly more careful in the successive steps of the above proof. Instead of considering cone sites to distance $\eta N$, for some fixed $\eta > 0$, we can consider $(\frac{\pi}{2} / (1 + \delta), N^{\frac{12}{13} + \delta})$-cone sites. Then with high probability, there are at least $N^{\frac{9}{13} - \delta}$ of them (along $\dout_{[B]} H_N$). By considering the corresponding time $\ol t := t_c + N^{-\frac{9}{13} + 2 \delta}$, at which $L(\ol t) \ll N^{\frac{12}{13} + \delta}$ (using \eqref{eq:theta_L_exp}), this provides the lower bound $(1 - \bar{\ve}) \pi_1(L(\ol t))$, which is $\geq N^{-\frac{5}{52} - \delta}$ for all $N$ large enough (from the one-arm exponent $\alpha_1$ mentioned below \eqref{eq:L_equiv}).


Finally, we discuss briefly the case $\zeta = \infty$ of an infinite rate of ignition, i.e. clusters burning immediately when they touch the boundary. In this case, we need to change a little bit the definition of cone sites (note that they could not even exist otherwise): at time $t_c$, we replace the event $\calA_1(\calC_n^{(\alpha)}(v))$ by the event $\{v$ is vacant, and there exists an occupied path from one of the two neighbors of $v$ above it to distance $n\}$ (and remaining within $\calC^{(\alpha)}(v)$). Then it is easy to see that Lemmas~\ref{lem:cone_est} and \ref{lem:cone_number} hold with this modified notion of cone sites, as well as Lemma~\ref{lem:depth_fires}. This completes the proof.
\end{proof}

\subsection{Role of recoveries: proof of Theorem~\ref{thm:main_result2}} \label{sec:proof2}

We now explain how tools developed in \cite{BN2022} can be used to obtain the analog of Theorem~\ref{thm:main_result1} for the process with recoveries, i.e. Theorem~\ref{thm:main_result2}. In that recent work, which was strongly inspired by the earlier paper \cite{KMS2015} by Kiss, Manolescu and Sidoravicius, and generalizes substantially the results in that paper, we study the geometric impact of recoveries in forest fires. For our purpose, one specific result derived in \cite{BN2022} turns out to be sufficient.

\begin{proof}[Proof of Theorem~\ref{thm:main_result2}]
Most of the proof of Theorem~\ref{thm:main_result1} can be repeated, and we only mention the extra input that is needed (using the same notations as in that proof). We know from \eqref{eq:proof_main2} that at the time $\ul t > t_c$, which satisfies $L(\ul t)  \ll N^{\frac{4}{3} \beta'}$ as $N \to \infty$, there exists an occupied circuit in $H_{N-N^{1-\delta}} \setminus H_{N - N^{4 \beta'/3}}$ with high probability (w.h.p.). It is easy to see that, in addition, there is a $\ul t$-occupied circuit in $A_{N/4,N/2}$, that we denote by $\calO'$, as well as a $\ul t$-occupied path from the former circuit $\calO$ to $\din B_{N/4}$. So in particular, these two mentioned circuits $\calO$ and $\calO'$ are connected, so that the ignition which triggers the burning of $\calO$ (see the sentence below \eqref{eq:proof_main4}) will also cause $\calO'$ to burn, ``isolating'' $0$ in an island contained in $B_{N/2}$.

Now, for the process with recoveries, we can apply Theorem~5.7 of \cite{BN2022} to the burning of $\calO'$ (together with some obvious monotonicity, since $\ul t > t_c$). This allows us to conclude that w.h.p., that burning will keep $0$ isolated until, at least, some (universal) time $\hat{t} > t_c$ (which is the time $t_c + \mathfrak{d}$ provided by \cite{BN2022}). This completes the proof.
\end{proof}


\section{Consequences for Graf's forest fire process} \label{sec:connection_Graf}

Finally, we give a proof for Theorem~\ref{thm:main_resultH}. Recall that this result says, roughly speaking, that no infinite occupied clusters (and hence, ``no infinite fires'') emerge in the forest fire process in the upper half-plane, with ignitions along the real line. As mentioned earlier, we focus on the process without recoveries in the present paper. We believe that the same result holds true for the original process, with recoveries, introduced by Graf, again up to some universal time $\hat{t} > t_c$. However, proving it seems to require some non-trivial adaptation of the results from \cite{KMS2015} and \cite{BN2022} (contrary to the proof of Theorem~\ref{thm:main_result2} above), so we decided not to include it in the present paper.

Strictly speaking, our proof below gives Theorem~\ref{thm:main_resultH} \emph{under the assumption that the process exists}. However, careful inspection of the proof shows that it also implies the analog of Theorem~\ref{thm:main_resultH} for the case of Graf's process without recoveries (which has the \emph{extra rule} that infinite clusters burn immediately, and for which existence follows from the same arguments as in \cite{Gr2014}). But then it automatically follows that our process also exists (and satisfies Theorem~\ref{thm:main_resultH}).

\begin{proof}[Proof of Theorem~\ref{thm:main_resultH}]
Choose any vertex $v \in V_{\HH}$, and let $k := \|v\|+1$, so that $v \in B_k$. Consider an arbitrary $\bar{\ve} > 0$. We follow the construction depicted on Figure~\ref{fig:half_plane}.

\begin{figure}[t]
\begin{center}

\includegraphics[width=.75\textwidth]{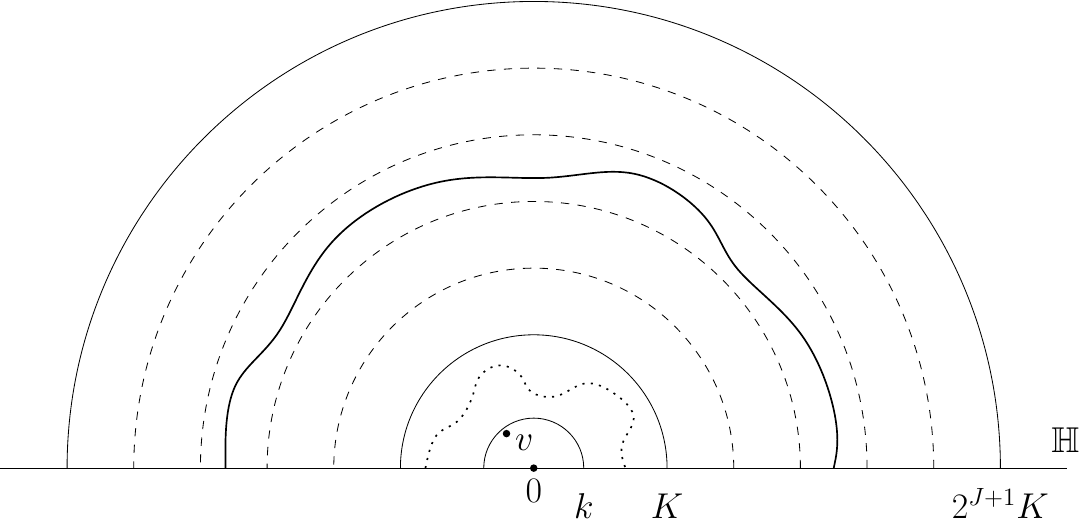}
\caption{\label{fig:half_plane} Let $k$ be such that $v \in B_k$, we let $K \geq k$ large enough so that (w.h.p.) there exists a $t_c$-vacant semi-circuit (in dotted line) in $A_{k,K} \cap V_{\HH}$, and then $J$ sufficiently large so that (again, w.h.p.) there exists a $t_c$-occupied circuit (in solid) in one of the semi-annuli $A_{2^j K, 2^{j+1} K} \cap V_{\HH}$, $1 \leq j \leq J$, which burns ``almost immediately''.}

\end{center}
\end{figure}

First, it follows from the classical RSW bounds at criticality that we can fix $K(k)$ large enough so that (in the pure birth process)
\begin{equation} \label{eq:pf_hp_1}
\PP \big( \text{at time $t_c$, } \calO^*(A_{k,K} \cap V_{\HH}) \text{ occurs} \big) \geq 1- \frac{\bar{\ve}}{3}
\end{equation}
(with obvious notation for the event $\calO^*$ in the semi-annulus $A^+_{k,K} := A_{k,K} \cap V_{\HH}$). We then let $\ul t$ be larger than $t_c$, but sufficiently close to it, so that
\begin{equation} 
\big| A^+_{k,K} \big| \cdot \big( 1 - e^{-(\ul t - t_c)} \big) \leq \frac{\bar{\ve}}{3}.
\end{equation}
Obviously, this requirement implies that
\begin{equation} \label{eq:pf_hp_2}
\PP \big( \text{no vertex of } A^+_{k,K} \text{ switches from vacant to occupied during } (t_c, \ul t) \big) \geq 1 - \frac{\bar{\ve}}{3}.
\end{equation}

From now on, we assume, without loss of generality, that $\zeta \in (0,\infty)$ (the case $\zeta = \infty$ can be handled through the same small change as we did toward the end of the proof of Theorem~\ref{thm:main_result1}). Using the existence of cone sites, similarly to the conclusion of Lemma~\ref{lem:cone_number} (with, e.g., the specific opening angle $\alpha = \frac{\pi}{3}$) and standard RSW bounds, it is easy to see that the following holds. For some $c_1(\zeta) > 0$, we have (for the process with marks): for all $j \geq 1$,
\begin{equation}
\PP \big( \text{at time $t_c$, there exists an occupied (unmarked) semi-circuit in } A^+_{2^j K, 2^{j+1}K} \big) \geq c_1.
\end{equation}
Moreover, the events appearing inside this probability can be ``made independent'' for odd values of $j$, using again localized versions of cone sites (as for the proof of Lemma~\ref{lem:cone_number}, see in particular Figure~\ref{fig:local_event}). Hence, there exists $J$ large enough so that
\begin{equation} \label{eq:pf_hp_3}
\PP \big( \text{during $(t_c,\ul t)$, a $t_c$-occupied semi-circuit burns in } A^+_{2^j K, 2^{j+1}K}, \text{ for some } j \in \{1,\ldots, J\} \big) \geq 1- \frac{\bar{\ve}}{3}.
\end{equation}

We observe that if the three events appearing in \eqref{eq:pf_hp_1}, \eqref{eq:pf_hp_2} and \eqref{eq:pf_hp_3} occur simultaneously, which has a probability at least $1 - \bar{\ve}$, then the following happens.
\begin{enumerate}[(1)]
\item Before time $\ul t$, the vertex $v$ is disconnected from infinity by a vacant semi-circuit, provided by $\calO^*(A^+_{k,K})$.

\item From time $\ul t$ on (and most likely, much earlier), $v$ is disconnected from infinity by a burnt semi-circuit.
\end{enumerate}
Hence, the occupied cluster of $v$ remains bounded over the whole time interval $[0,\infty)$, with probability at least $1 - \bar{\ve}$. Since $\bar{\ve}$ can be taken arbitrarily small, we finally get that
$$\PP \big( \text{the occupied cluster of $v$ remains bounded} \big) = 1.$$
This completes the proof, using the countability of $V_{\HH}$.
\end{proof}

\bibliographystyle{plain}
\bibliography{FF_boundary}

\end{document}